\documentclass[reqno]{amsart}
\usepackage{amsmath, amsthm, amssymb, amstext}

\usepackage{hyperref,xcolor}
\hypersetup{
 pdfborder={0 0 0},
 colorlinks,
}

\usepackage{enumitem}
\newtheorem{theorem}{Theorem}
\newtheorem{remark}[theorem]{Remark}

\newtheorem{proposition}[theorem]{Proposition}
\newtheorem{corollary}[theorem]{Corollary}

\DeclareMathOperator*{\divergenz}{div}              %
\DeclareMathOperator*{\ints}{int}         %
\DeclareMathOperator*{\ww}{w}         %
\DeclareMathOperator*{\Ss}{S}         %

\newcommand{\N}{\mathbb{N}}

\newcommand{\R}{\mathbb{R}}

\newcommand{\RN}{\mathbb{R}^N}

\newcommand{\Lp}[1]{L^{#1}(\Omega)}

\newcommand{\Wpzero}[1]{W^{1,#1}_0(\Omega)}

\newcommand{\lan}{\langle}
\newcommand{\ran}{\rangle}
\newcommand{\eps}{\varepsilon}
\newcommand{\ph}{\varphi}
\newcommand{\Om}{\Omega}

\newcommand{\into}{\int_{\Omega}}
\newcommand{\weak}{\overset{\ww}{\to}}

\newcommand{\Linf}{L^{\infty}(\Omega)}
\newcommand{\close}{\overline{\Omega}}
\newcommand{\interior}{\ints \left(C^1_0(\overline{\Omega})_+\right)}

\newcommand{\cprime}{$'$}

\renewcommand{\l}{\left}
\renewcommand{\r}{\right}

\numberwithin{theorem}{section}
\numberwithin{equation}{section}

\title[Existence/nonexistence of positive solutions for singular equations]{Existence and nonexistence of positive solutions for singular $(p,q)$-equations with superdiffusive perturbation}

\author[N.\,S.\,Papageorgiou]{Nikolaos S.\,Papageorgiou}
\address[N.\,S.\,Papageorgiou]{National Technical University, Department of Mathematics, Zografou Campus, Athens 15780, Greece}
\email{npapg@math.ntua.gr}

\author[P.\,Winkert]{Patrick Winkert}
\address[P.\,Winkert]{Technische Universit\"{a}t Berlin, Institut f\"{u}r Mathematik, Stra\ss e des 17.\,Juni 136, 10623 Berlin, Germany}
\email{winkert@math.tu-berlin.de}

\subjclass[2010]{35J75, 35J92}
\keywords{$(p,q)$-Laplacian, superdiffusive perturbation, positive solutions, nonlinear regularity, truncation and comparison methods}

\begin{document}

\begin{abstract}
	We consider a nonlinear Dirichlet problem driven by the $(p,q)$-Laplacian and with a reaction which is parametric and exhibits the combined effects of a singular term and of a superdiffusive one. We prove an existence and nonexistence result for positive solutions depending on the value of the parameter $\lambda \in \overset{\circ}{\R}_+=(0,+\infty)$.
\end{abstract}

\maketitle

\section{Introduction}

Let $\Omega\subseteq \R^N$ be a bounded domain with a $C^2$-boundary $\partial \Omega$. In this paper, we study the following singular $(p,q)$-equation with logistic perturbation
\begin{equation}\tag{P$_\lambda$}\label{problem}
    \begin{aligned}
	-\Delta_p u-\Delta_q u   & =\lambda \l [u^{-\eta}+u^{\theta-1}\r] -f(x,u)\quad && \text{in } \Omega,\\
	u & = 0  &&\text{on } \partial \Omega,\\
	u>0, \quad \lambda&>0, \quad 0<\eta<1, \quad 1<q<p<\theta.&&
    \end{aligned}
\end{equation}

For $r\in (1,\infty)$ we denote the $r$-Laplace differential operator defined by
\begin{align*}
	\Delta_r u =\divergenz \left(|\nabla u|^{r-2} \nabla u\right)\quad\text{for all }u \in \Wpzero{r}.
\end{align*}

In problem \eqref{problem} we have the sum of two such operators with different exponents which implies that the differential operator on the left-hand side is not homogeneous. The right-hand side of \eqref{problem} has the combined effects of a singular term $s \to \lambda s^{-\eta}$ for $s>0$ with $0<\eta<1$ and of a perturbation which is of logistic type, namely the function $s \to \lambda s^{\theta-1}-f(x,s)$ for almost all (a.\,a.) $x\in\Omega$. The function $f\colon\Omega\times\R\to\R$ is a Carath\'{e}odory function, that is, $x\mapsto f(x,s)$ is measurable for all $s\in \R$ and $s\mapsto f(x,s)$ is continuous for a.\,a.\,$x\in \Omega$. We assume that $f(x,\cdot)$ is $(\theta-1)$-superlinear as $s \to +\infty$ for a.\,a.\,$x\in\Omega$. So, the logistic perturbation is of the superdiffusive type. We are interested in positive solutions whenever the parameter $\lambda$ is positive.

Parametric superdiffusive logistic equations with no singular term present, were investigated by Afrouzi-Brown \cite{1-Afrouzi-Brown-1998} (for semilinear Dirichlet problems), Takeuchi \cite{23-Takeuchi-2001}, \cite{24-Takeuchi-2001b} (for nonlinear Dirichlet problems driven by the $p$-Laplacian), Gasi\'{n}ski-O'Regan-Papageorgiou \cite{3-Gasinski-ORegan-Papageorgiou-2015} (for nonlinear Dirichlet problems driven by a nonhomogeneous differential operator), Cardinali-Papageorgiou-Rubbioni \cite{2-Cardinali-Papageorgiou-Rubbioni-2014}, Gasi\'{n}ski-Papageorgiou \cite{5-Gasinski-Papageorgiou-2017} (both dealing with nonlinear problems driven by the $p$-Laplacian) and Papageorgiou-R\u{a}dulescu-Repov\v{s} \cite{14-Papageorgiou-Radulescu-Repovs-2018} (for semilinear mixed problems). These works reveal that the superdiffusive logistic equations exhbit a kind of ``bifurcation'' for large values of the parameter $\lambda>0$. More precisely, there is a critical parameter value $\lambda_*>0$ such that the problem has at least two positive solutions for all $\lambda>\lambda_*$, the problem hast at least one positive solution for $\lambda=\lambda_*$ and there are no positive solutions for $\lambda \in (0,\lambda_*)$. This is in contrast to subdiffusive and equidiffusive logistic equations for which we do not have multiplicity of positive solutions, see Papageorgiou-Winkert \cite{18-Papageorgiou-Winkert-2014}.

When we introduce a singular term in the reaction, the geometry of the problem changes since $u=0$ is no longer a local minimizer of the energy functional and so we cannot have a multiplicity result. In addition, the singular term generates an energy functional which is not $C^1$ and so we have difficulties in using the results of critical point theory. Therefore, we need to find a way to bypass the singular term and deal with a $C^1$-functional to which we can apply the results of the critical point theory. Nonlinear singular problems but with a different kind of perturbation were studied recently by Papageorgiou-Winkert \cite{20-Papageorgiou-Winkert-2019} (equations driven by the $p$-Laplacian) and by Papageorgiou-R\u{a}dulescu-Repov\v{s} \cite{16-Papageorgiou-Radulescu-Repovs-2020} (equations driven by a nonhomogeneous differential operator).

The main result of our work here establishes the existence of a critical parameter $\lambda_*$ such that
\begin{enumerate}
	\item[$\bullet$]
		problem \eqref{problem} has at least one positive smooth solution for all $\lambda \geq \lambda_*$;
	\item[$\bullet$]
		problem \eqref{problem} has no positive solutions for all $\lambda< \lambda_*$.
\end{enumerate}

Finally we mention that equations driven by the sum of two differential operators of different nature (such as ($p,q$)-equations) arise in many mathematical models of physical processes. We refer to the survey papers of Marano-Mosconi \cite{12-Marano-Mosconi-2018} and R\u{a}dulescu \cite{22-Radulescu-2019}.

\section{Preliminaries and Hypotheses}\label{section_2}

In this section we present some preliminaries which are needed in the sequel and also the hypotheses on the data of problem \eqref{problem}. 

For every $1 \le r< \infty$ we consider the usual Lebesgue spaces $L^r(\Om)$ and $L^r(\Om; \RN)$ equipped with the norm $\|\cdot \|_r$. When $1< r< \infty$ we denote by $W^{1, r}(\Om)$ and $W^{1,r}_0(\Om)$ the corresponding Sobolev spaces  equipped with the norms $\|\cdot \|_{1,r}$ and $\|\cdot \|_{1,r,0}$, respectively. Because of the Poincar\'{e} inequality we can equip the space $\Wpzero{r}$ with the following norm
\begin{align*}
    \|u\|=\|\nabla u\|_r \quad\text{for all }u\in\Wpzero{r},
\end{align*}
The Banach space
\begin{align*}
   C^1_0(\overline{\Omega})= \left\{u \in C^1(\overline{\Omega})\,:\, u\big|_{\partial \Omega}=0 \right\}
\end{align*}
is an ordered Banach space with positive cone
\begin{align*}
   C^1_0(\overline{\Omega})_+=\left\{u \in C^1_0(\overline{\Omega})\,:\, u(x) \geq 0 \text{ for all } x \in \overline{\Omega}\right\}.
\end{align*}
This cone has a nonempty interior given by
\begin{align*}
   \ints \left(C^1_0(\overline{\Omega})_+\right)=\left\{u \in C^1_0(\overline{\Omega})_+: u(x)>0 \text{ for all } x \in \Omega \text{, } \frac{\partial u}{\partial n}(x)<0 \text{ for all } x \in \partial \Omega \right\},
\end{align*}
where $n(\cdot)$ stands for the outward unit normal on $\partial \Omega$.

Let $r\in (1,+\infty)$ and recall that $\Wpzero{r}^*=W^{-1,r'}(\Omega)$ with $\frac{1}{r}+\frac{1}{r'}=1$. By $\lan \cdot,\cdot\ran_{1,r}$ we denote the duality brackets of the pair $(\Wpzero{r},W^{-1,r'}(\Omega))$. For notational simplicity when $r=p$, we simply write $\lan\cdot,\cdot\ran$. 

For $r\in (1,+\infty)$, let $A_r\colon\Wpzero{r}\to W^{-1,r'}(\Omega)=\Wpzero{r}^*$ with $\frac{1}{r}+\frac{1}{r'}=1$ be the nonlinear map defined by
\begin{align}\label{p-Laplace}
    \langle A_r(u), h\rangle_{1,r}=\into |\nabla u|^{r-2}\nabla u \cdot \nabla h\,dx \quad\text{for all }u,h\in\Wpzero{r}.
\end{align}

From Gasi{\'n}ski-Papageorgiou \cite[Problem 2.192, p.\,279]{7-Gasinski-Papageorgiou-2016} we have the following properties of $A_r$.

\begin{proposition}\label{proposition_1}
    The map $A_r\colon\Wpzero{r}\to W^{-1,r'}(\Omega)$ defined in \eqref{p-Laplace} is bounded, that is, it maps bounded sets to bounded sets, continuous, strictly monotone, hence maximal monotone and it is of type $(\Ss)_+$, that is,
    \begin{align*}
	u_n \weak u \text{ in }\Wpzero{r}\quad\text{and}\quad \limsup_{n\to\infty} \langle A_r(u_n),u_n-u\rangle \leq 0,
    \end{align*}
    imply $u_n\to u$ in $\Wpzero{r}$.
\end{proposition}

For $s \in \R$, we set $s^{\pm}=\max\{\pm s,0\}$ and for $u \in W^{1,p}_0(\Omega)$ we define $u^{\pm}(\cdot)=u(\cdot)^{\pm}$. It is well known that
\begin{align*}
    u^{\pm} \in W^{1,p}_0(\Omega), \quad |u|=u^++u^-, \quad u=u^+-u^-.
\end{align*}

Furthermore, given a measurable function $g\colon \Omega\times\R\to\R$, we denote by $N_g$ the corresponding Nemytskii (superposition) operator defined by
\begin{align*}
	N_g(u)(\cdot)=g(\cdot,u(\cdot))\quad\text{for all  measurable }u\colon \Omega\to \R.
\end{align*}
It is clear that $x\to g(x,u(x))$ is measurable. Recall that if $g\colon\Omega\times\R\to\R$ is a Carath\'eodory function, then $g$ is measurable in both arguments, see, for example, Papageorgiou-Winkert \cite[Proposition 2.2.31, p.\,106]{19-Papageorgiou-Winkert-2018}.

If $h_1,h_2\colon\Omega\to\R$ are two measurable functions, then we write $h_1\prec h_2$ if and only if for every compact $K\subseteq\Omega$ we have $0<c_K\leq h_2(x)-h_1(x)$ for a.\,a.\,$x\in K$.  Note that if $h_1,h_2\in C(\Omega)$ and $h_1(x) < h_2(x)$ for all $x \in \Omega$, then $h_1 \prec h_2$.

For $u,v\in\Wpzero{p}$ with $u(x)\leq v(x)$ for a.\,a.\,$x\in\Omega$ we define
\begin{align*}
    [u,v]&=\big\{h\in\Wpzero{p}: u(x)\leq h(x)\leq v(x)\text{ for a.\,a.\,}x\in\Omega\big\},\\
    [u)&=\big\{h\in\Wpzero{p}: u(x)\leq h(x)\text{ for a.\,a.\,}x\in\Omega\big\}.
\end{align*}

Now we are ready to introduce the hypotheses on the perturbation $f\colon\Omega\times\R\to\R$.
\begin{enumerate}[leftmargin=1.2cm]
	\item[H:]
		$f\colon\Omega \times \R\to \R$ is a Carath\'{e}odory function such that, for a.\,a.\,$x\in\Omega$, $f(x,0)=0$, $f(x,\cdot)$ is nondecreasing and
		\begin{enumerate}[itemsep=0.2cm, topsep=0.2cm]
			\item[(i)]
				\begin{align*}
		    		f(x,s)\leq a(x) \left(1+s^{r-1}\right)
				\end{align*}
				for a.a.\,$x\in\Omega$, for all $s\geq 0$, with $a\in \Linf$ and $\theta<r<p^*$, where $p^*$ denotes the critical Sobolev exponent with respect to $p$ given by
				\begin{align*}
					p^*=
					\begin{cases}
						\frac{Np}{N-p} & \text{if }p<N,\\
						+\infty & \text{if } N \leq p;
		    			\end{cases}
				\end{align*}
			\item[(ii)]
				\begin{align*}
		    		\lim_{s\to +\infty} \frac{f(x,s)}{s^{\theta-1}}=+\infty\quad\text{uniformly for a.\,a.\,}x\in\Omega;
				\end{align*}
			\item[(iii)]
				there exist $0<\hat{\eta}_1 \leq \hat{\eta}_2$ and $\delta_0>0$ such that
				\begin{align*}
		    		\hat{\eta}_1s^{q-1} \leq f(x,s) \quad \text{for a.\,a.\,}x\in \Omega \text{ and for all } s\in [0,\delta_0]
				\end{align*}
				and
				\begin{align*}
					\limsup_{s\to 0^+} \frac{f(x,s)}{s^{q-1}} \leq \hat{\eta}_2\quad\text{uniformly for a.\,a.\,}x\in\Omega.
				\end{align*}
	\end{enumerate}
\end{enumerate}

\begin{remark}
	With view to our problem it is clear that we are looking for positive solutions and the hypotheses above only concern the positive semiaxis $\R_+=[0,+\infty)$. Therefore, without any loss generality, we may assume that
	\begin{align*}
		f(x,s)=0\quad\text{for a.a.\,}x\in \Omega\text{ and for all }s\leq 0.
	\end{align*}
	Hypothesis H(ii) implies that $f(x,\cdot) $ is $(\theta-1)$-superlinear as $s \to +\infty$ for a.\,a.\,$x\in\Omega$. Dropping the $x$-dependence for simplicity, the following functions satisfy hypotheses $H$
	\begin{align*}
		f_1(x)&=
		\begin{cases}
			\l(s^+\r)^{q-1} &\text{if }  s \leq 1,\\
			s^{\theta-1} \l[\ln(x)+1\r] &\text{if } 1<s,
		\end{cases}\quad\text{with } 1<q<p<\theta<p^*,\\[1ex]
		f_2(x)&=
		\begin{cases}
			\mu \l(s^+\r)^{q-1}-\l(s^+\r)^{\tau-1} &\text{if } s \leq 1,\\
			(\mu-1)s^{r-1}&\text{if } 1<s
		\end{cases}\quad\text{with } 1<q<p<r<p^*,
	\end{align*}
	and $\tau>q$ as well as $\mu\geq \frac{p-1}{q-1}$.
\end{remark}

As we already mentioned in the Introduction, the presence of the singular term leads to an energy functional which is not $C^1$. This creates problems in the usage of variational tools. In the next section we examine an auxiliary singular problem and the solution of them will help us in order to avoid difficulties of having to do with a nonsmooth energy functional.

\section{An auxiliary singular problem} 

In this section we deal with the following parametric singular Dirichlet $(p,q)$-equation

\begin{equation}\tag{Q$_\lambda$}\label{problem2}
	\begin{aligned}
		-\Delta_p u-\Delta_q u   & =\lambda u^{-\eta}-f(x,u)\quad && \text{in } \Omega,\\
		u & = 0  &&\text{on } \partial \Omega,\\
		u>0, \quad \lambda&>0, \quad 0<\eta<1, \quad 1<q<p.&&
	\end{aligned}
\end{equation}

For this problem we have the following existence and uniqueness result.

\begin{proposition}\label{proposition_2}
	If hypotheses $H$ hold, then for every $\lambda>0$, problem \eqref{problem2} has a unique  positive solution $\overline{u}_\lambda\in\interior$ and the map $\lambda \to \overline{u}_\lambda$ is nondecreasing from $\overset{\circ}{\R}_+=(0,+\infty)$ into $C^1_0(\close)$.
\end{proposition}

\begin{proof}
	First we show the existence of a positive solution for problem \eqref{problem2} for every $\lambda>0$.
	
	To this end, let $g \in \Lp{p}$ and $\eps>0$. We consider the following Dirichlet problem 
	\begin{equation*}
		\begin{aligned}
			-\Delta_p u-\Delta_q u +f(x,u)  & = \frac{\lambda}{\l[|g|+\eps\r]^{\eta}}\quad && \text{in } \Omega,\\[1ex]
			u & = 0  &&\text{on } \partial \Omega,
		\end{aligned}
	\end{equation*}

	Moreover, we consider the nonlinear operator $V\colon \Wpzero{p}\to W^{-1,p'}(\Omega)$ defined by
	\begin{align*}
		V(u)=A_p(u)+A_q(u)+N_f(u)\quad\text{for all }u \in \Wpzero{p}.
	\end{align*}
	Recall that $\Wpzero{p} \hookrightarrow \Wpzero{q}$ continuously and densely implies that $W^{-1,q'}(\Omega)$ $\hookrightarrow W^{-1,p'}(\Omega)$ continuously and densely as well, see Gasi\'nski-Papageorgiou \cite[Lemma 2.2.27, p.\,141]{4-Gasinski-Papageorgiou-2006}.

	By Proposition \ref{proposition_1} and the fact that $f(x,\cdot)$ is nondecreasing, we know that $V\colon \Wpzero{p}\to W^{-1,p'(\Omega)}$ is continuous and strictly monotone, hence, maximal monotone as well. In addition we have
	\begin{align*}
		\lan V(u),u\ran \geq \lan A_p(u),u\ran =\|\nabla u\|^p_p=\|u\|^p \quad\text{for all }u\in\Wpzero{p},
	\end{align*} 
	which implies that $V\colon \Wpzero{p}\to W^{-1,p'(\Omega)}$ is also coercive. Therefore, it is surjective, see Papageorgiou-R\u{a}dulescu-Repov\v{s} \cite[Corollary 2.8.7, p.\,135]{15-Papageorgiou-Radulescu-Repovs-2019}. Note that
	\begin{align*}
		\frac{\lambda}{\l[|g(\cdot)|+\eps\r]^{\eta}} \in \Linf \hookrightarrow W^{-1,p'}(\Omega).
	\end{align*}
	Hence, there exists $v_\eps\in\Wpzero{p}$ such that
	\begin{align*}
		V(v_\eps)=\frac{\lambda}{\l[|g|+\eps\r]^{\eta}}.
	\end{align*}
	The strict monotonicity of $V$ implies that this solution $v_\eps$ is unique. Since $\Wpzero{p}$ $\hookrightarrow\Lp{p}$ by the Sobolev embedding theorem, we can define the solution map $k_\eps\colon \Lp{p}$ $\to \Lp{p}$ by $k_\eps(g)=v_\eps$. Note that
	\begin{align}\label{2}
		A_p(v_\eps)+A_q(v_\eps)+N_f(v_\eps)=\frac{\lambda}{\l[|g|+\eps\r]^{\eta}} \quad\text{in } W^{-1,p'}(\Omega).
	\end{align}
	On \eqref{2} we take the test function $v_\eps \in\Wpzero{p}$ and obtain
	\begin{align}\label{3}
		\l\|\nabla v_\eps\r\|_p^p = \l\|v_\eps\r\|^p \leq \frac{\lambda}{\eps^\eta}
	\end{align}
	because $f(x,v_\eps)v_\eps \geq 0$. From the compactness of $\Wpzero{p} \hookrightarrow\Lp{p}$ it follows that
	\begin{align*}
		\overline{k_\eps(\Lp{p})}^{\|\cdot\|_p} \subseteq \Lp{p} \text{ is compact}.
	\end{align*}

	Suppose that $g_n \to g$ in $\Lp{p}$. From \eqref{3} we see that
	\begin{align*}
		\l\{v^n_\eps\r\}_{n\in\N}=\l\{k_\eps(g_n)\r\}_{n\in\N} \subseteq \Wpzero{p} \text{ is bounded}.
	\end{align*}
	Hence, by passing to a suitable subsequence if necessary, we may assume that
	\begin{align}\label{4}
		v^n_\eps \weak v_\eps^* \quad\text{in }\Wpzero{p}\quad\text{and}\quad v^n_\eps \to v^*_n \quad\text{in }\Lp{p}.
	\end{align}
	We have
	\begin{align}\label{5}
		A_p\l(v^n_\eps\r)+A_q\l(v^n_\eps\r)+N_f\l(v^n_\eps\r)=\frac{\lambda}{\l[|g_n|+\eps\r]^{\eta}}\quad\text{in }W^{-1,p'}(\Omega)
	\end{align} 
	for all $n \in \N$. Applying $v^n_\eps-v^*_\eps\in\Wpzero{p}$ on \eqref{5}, passing to the limit as $n\to \infty$ and using \eqref{4}, we obtain
	\begin{align*}
		\lim_{n\to \infty} \l[\l\lan A_p\l(v^n_\eps\r), v^n_\eps-v^*_\eps\r\ran+\l\lan A_q\l(v^n_\eps\r),v^n_\eps-v^*_\eps \r\ran\r]=0.
	\end{align*}
	Since $A_q$ is monotone, we derive
	\begin{align*}
		\limsup_{n\to \infty} \l[\l\lan A_p\l(v^n_\eps\r), v^n_\eps-v^*_\eps\r\ran+\l\lan A_q\l(v^*_\eps\r),v^n_\eps-v^*_\eps \r\ran\r]\leq 0
	\end{align*}
	and due to \eqref{4}, we get
	\begin{align*}
		\limsup_{n\to \infty} \l\lan A_p\l(v^n_\eps\r), v^n_\eps-v^*_\eps\r\ran\leq 0.
	\end{align*}
	Then, by Proposition \ref{proposition_1}, it follows that
	\begin{align}\label{6}
		v^n_\eps \to v_\eps^* \text{ in }\Wpzero{p}.
	\end{align}
	
	So, if we pass in \eqref{5} to the limit as $n\to \infty$ and use \eqref{6} as well as the fact that $|g_n|\to |g|$ in $\Lp{p}$, we obtain
	\begin{align*}
		A_p\l(v^*_\eps\r)+A_q\l(v^*_\eps\r)+N_f\l(v^*_\eps\r)=\frac{\lambda}{\l[|g|+\eps\r]^{\eta}}\quad\text{in }W^{-1,p'}(\Omega).
	\end{align*}
	Hence, $v^*_\eps=k_\eps(g)$.
	
	By the Urysohn's criterion for the convergence of sequences we have for the initial sequence $k_\eps(g_n)\to k_\eps (g)$ in $\Lp{p}$, see Gasi\'nski-Papageorgiou \cite[p.\,33]{6-Gasinski-Papageorgiou-2014}. This proves that the solution map $k_\eps$ is continuous. Therefore, we can apply the Schauder-Tychonov fixed point theorem, see Papageorgiou-R\u{a}dulescu-Repov\v{s} \cite[Theorem 4.3.21, p.\,298]{15-Papageorgiou-Radulescu-Repovs-2019}, which gives the existence of $\hat{v}_\eps\in\Wpzero{p}$ such that
	\begin{align*}
		k_\eps\l(\hat{v}_\eps\r)=\hat{v}_\eps, \quad \hat{v}_\eps \geq 0, \quad \hat{v}_\eps\neq 0.
	\end{align*}

	We have
	\begin{equation*}
	\begin{aligned}
	-\Delta_p \hat{v}_\eps-\Delta_q \hat{v}_\eps & = \frac{\lambda}{\l[\hat{v}_\eps+\eps\r]^{\eta}}-f\l(x,\hat{v}_\eps\r)\quad && \text{in } \Omega,\\[1ex]
	\hat{v}_\eps & = 0  &&\text{on } \partial \Omega.
	\end{aligned}
	\end{equation*}
	Theorem 7.1 of Ladyzhenskaya-Ural'tseva \cite[p.\,286]{9-Ladyzhenskaya-Uraltseva-1968} implies that $\hat{v}_\eps \in \Linf$. Then, from the nonlinear regularity theory of Lieberman \cite{11-Lieberman-1991} we have that $\hat{v}_\eps\in C^1_0(\close)_+\setminus \{0\}$. Hypotheses $H$(i), (iii) imply that if $\rho_\eps=\|\hat{v}_\eps\|_\infty$, then there exists $\hat{\xi}_{\rho_\eps}>0$ such that $\hat{\xi}_{\rho_\eps} s^{p-1}-f(x,s) \geq 0$ for a.\,a.\,$x\in\Omega$ and for all $s\in[0,\rho_\eps]$. Using this we obtain
	\begin{align*}
		-\Delta_p \hat{v}_\eps-\Delta_q \hat{v}_\eps+\hat{\xi}_{\rho_\eps} \hat{v}^{p-1} \geq 
		\hat{\xi}_{\rho_\eps} \hat{v}^{p-1}-f\l(x,\hat{v}_\eps\r) \geq 0 \quad \text{in }\Omega.
	\end{align*}
	Hence, we have
	\begin{align*}
		\Delta_p \hat{v}_\eps+\Delta_q \hat{v}_\eps \leq \hat{\xi}_{\rho_\eps} \hat{v}^{p-1},
	\end{align*}
	which implies that $\hat{v}_\eps \in \interior$, see Pucci-Serrin \cite[pp.\,111 and 120]{21-Pucci-Serrin-2007}.
	
	Therefore, we produced a solution $\hat{v}_\eps \in \interior$ for the following approximation of problem \eqref{problem2}
	\begin{equation}\label{7}
		\begin{aligned}
			-\Delta_p u-\Delta_q u  & = \frac{\lambda}{\l[|u|+\eps\r]^{\eta}}-f(x,u)\quad && \text{in } \Omega,\\[1ex]
			u\big|_{\partial\Omega} & = 0, \quad u\geq 0.  &&
		\end{aligned}
	\end{equation}
	
	In fact this solutions is unique. Indeed, if $\tilde{v}_\eps \in \Wpzero{p}$ is another positive solution of \eqref{7}, then we have
	\begin{align*}
		0 &
		\leq \l\lan A_p\l(\hat{v}_\eps\r)-A_p\l(\tilde{v}_\eps\r), \hat{v}_\eps-\tilde{v}_\eps\r\ran+\l\lan A_q\l(\hat{v}_\eps\r)-A_q\l(\tilde{v}_\eps\r), \hat{v}_\eps-\tilde{v}_\eps\r\ran\\
		& \quad + \into \l[f\l(x,\hat{v}_\eps\r)-f\l(x,\tilde{v}_\eps\r)\r] \l(\hat{v}_\eps-\tilde{v}_\eps\r)\,dx\\
		& = \into \lambda \l[\frac{1}{\l(\hat{v}_\eps+\eps\r)^\eta}-\frac{1}{\l(\tilde{v}_\eps+\eps\r)^\eta}\r] \l(\hat{v}_\eps-\tilde{v}_\eps\r)\,dx \leq 0.
	\end{align*}
	Since $u \to A_p(u)+A_q(u)$ is strictly monotone, see Proposition \ref{proposition_1}, it follows that $\hat{v}_\eps=\tilde{v}_\eps$. This proves the uniqueness of the solution $\hat{v}_\eps \in \interior$ of \eqref{7}.
	
	{\bf Claim: } If $0<\eps'<\eps\leq 1$, then $\hat{v}_\eps \leq \hat{v}_{\eps'}$.
	
	We have
	\begin{align}\label{8}
		-\Delta_p \hat{v}_{\eps'}-\Delta_q \hat{v}_{\eps'}+f\l(x,\hat{v}_{\eps'}\r)=\frac{\lambda}{\l[\hat{v}_{\eps'}+\eps'\r]^\eta}\geq \frac{\lambda}{\l[\hat{v}_{\eps'}+\eps\r]^\eta}\quad\text{in }\Omega.
	\end{align}
	
	Now we introduce the Carath\'eodory function $e_\eps\colon\Omega\times\R\to\R$ defined by
	\begin{align}\label{9}
		e_\eps=
		\begin{cases}
			\frac{\lambda}{\l[s^++\eps\r]^\eta} &\text{if }s\leq \hat{v}_{\eps'}(x),\\[1ex]
			\frac{\lambda}{\l[\hat{v}_{\eps'}(x)+\eps\r]^\eta}&\text{if }\hat{v}_{\eps'}(x)<s.
		\end{cases}
	\end{align}
	We set $E_\eps(x,s)=\int_0^se_\eps(x,t)\,dt$ and consider the $C^1$-functional $\sigma_\eps\colon\Wpzero{p}\to\R$ defined by
	\begin{align*}
		\sigma_\eps(u)=\frac{1}{p} \|\nabla u\|^p_p+\frac{1}{q}\|\nabla u\|^q_q + \into F\l(x,u^+\r)\,dx-\into E_\eps(x,u)\,dx
	\end{align*}
	for all $u\in\Wpzero{p}$. From \eqref{9} and since $F\geq 0$ we see that $\sigma_\eps\colon\Wpzero{p}\to\R$ is coercive and because of the Sobolev embedding theorem it is also sequentially weakly lower semicontinuous. Therefore, by the Weierstra\ss-Tonelli theorem there exists $\tilde{v}_\eps\in\Wpzero{p}$ such that
	'\begin{align*}
		\sigma_\eps\l(\tilde{v}_\eps\r)=\min\l[\sigma_\eps(v)\,:\,v\in\Wpzero{p}\r].
	\end{align*}
	This implies that $\sigma_\eps'\l(\tilde{v}_\eps\r)=0$, that is,
	\begin{align}\label{10}
		\l\lan A_p\l(\tilde{v}_\eps\r) ,h\r\ran+\l\lan A_q\l(\tilde{v}_\eps\r) ,h\r\ran+\into f\l(x,\tilde{v}_\eps\r)h\,dx=\into e_\eps\l(x,\tilde{v}_\eps\r)h\,dx
	\end{align}
	for all $h\in\Wpzero{p}$. Taking $h=-\tilde{v}_\eps^-\in\Wpzero{p}$ as test function in \eqref{10} and applying \eqref{9} we obtain that $\tilde{v}_\eps \geq 0$. Moreover, we can choose $h=\l(\tilde{v}_\eps-\hat{v}_{\eps'}\r)^+\in\Wpzero{p}$. Then, using once again \eqref{9} and also \eqref{8} we infer that $\tilde{v}_\eps \leq \hat{v}_{\eps'}$. So, we have proved that
	\begin{align}\label{11}
		\tilde{v}_\eps \in \l[0,\hat{v}_{\eps'}\r].
	\end{align}
	From \eqref{11}, \eqref{9} and \eqref{10} it follows that
	\begin{equation*}
	\begin{aligned}
		-\Delta_p \tilde{v}_\eps-\Delta_q \tilde{v}_\eps +f\l(x,\tilde{v}_\eps \r)& =\l[\tilde{v}_\eps+\eps\r]^{-\eta} \quad && \text{in } \Omega,\\[1ex]
		\tilde{v}_\eps\big|_{\partial\Omega} & = 0, \quad \tilde{v}_\eps\geq 0.  &&
	\end{aligned}
	\end{equation*}
	It is clear that $\tilde{v}_\eps\neq 0$ and so from the first part of the proof we have $\tilde{v}_\eps=\hat{v}_\eps \in\interior$. Then, due to \eqref{11}, we obtain $\tilde{v}_\eps\leq \tilde{v}_{\eps'}$. This proves the Claim.

	Now we are ready to send $\eps \to 0^+$ in order to produce a solution for problem \eqref{problem2}. So, we consider a sequence $\eps_n\to 0^+$ and set $\hat{v}_n=\hat{v}_{\eps_n}$ for all $n \in \N$. We have
	\begin{align}\label{12}
		\l\lan A_p\l(\hat{v}_n\r),h\r\ran+\l\lan A_q\l(\hat{v}_n\r),h\r\ran+\into f\l(x,\hat{v}_n\r)h\,dx=\into \frac{\lambda h}{\l[\hat{v}_n+\eps_n\r]^\eta}\,dx
	\end{align}
	for all $h \in \Wpzero{p}$. Testing \eqref{12} with $h= \hat{v}_n\in\Wpzero{p}$ and applying the Claim gives
	\begin{align}\label{13}
		\l\|\hat{v}_n\r\|^p=\l\|\nabla \hat{v}_n\r\|^p_p \leq \into \frac{\lambda \hat{v}_n}{\l[\hat{v}_n+\eps_n\r]^\eta}\,dx\leq\into \lambda \hat{v}_n \hat{v}_1^{-\eta}\,dx
	\end{align}
	for all $n\in \N$.
	
	Let $\hat{d}(x)=d(x,\partial\Omega)$ for $x \in \close$. We know that $\hat{d}\in\interior$, see Gilbarg-Trudinger \cite[p.\,355]{8-Gilbarg-Trudinger-1998}. Since $\hat{v}_1\in\interior$, we have
	\begin{align}\label{hardy}
		\begin{split}
			\into \lambda \hat{v}_n \hat{v}_1^{-\eta}\,dx
			& =\into \lambda \hat{v}_1^{1-\eta} \frac{\hat{v}_n}{\hat{v}_1}dx
			\leq \lambda c_1 \into \frac{\hat{v}_n}{\hat{v}_1}\,dx
			\leq \lambda c_2 \into \frac{\hat{v}_n}{\hat{d}}\,dx\\
			&\leq \lambda c_3 \l\|\frac{\hat{v}_n}{\hat{d}}\r\|_p
			\leq \lambda c_4 \l\|\hat{v}_n\r\|
		\end{split}
	\end{align}
	for some $c_1,c_2,c_3,c_4>0$.

	From \eqref{13} and \eqref{hardy} it follows that $\{\hat{v}_n\}\subseteq \Wpzero{p}$ is bounded. Therefore we may assume that
	\begin{align}\label{15}
		\hat{v}_n \weak \overline{u}_\lambda \quad\text{in }\Wpzero{p}\quad\text{and}\quad \hat{v}_n\to \overline{u}_\lambda \quad\text{in } \Lp{r}.
	\end{align}

	Now we choose $h=\hat{v}_n-\overline{u}_\lambda\in\Wpzero{p}$ in \eqref{12}. This yields
	\begin{align*}
		&\l\lan A_p\l(\hat{v}_n\r),\hat{v}_n-\overline{u}_\lambda\r\ran+\l\lan A_q\l(\hat{v}_n\r),\hat{v}_n-\overline{u}_\lambda\r\ran +\into f\l(x,\hat{v}_n\r)\l(\hat{v}_n-\overline{u}_\lambda\r)\,dx \\
		&= \into \frac{\lambda \l[\hat{v}_n-\overline{u}_\lambda\r]}{\l[\hat{v}_n+\eps_n\r]^\eta}\,dx\\
		& \leq \into \lambda \l[\hat{v}_n-\overline{u}_\lambda\r]^{1-\eta}\,dx\\
		& \leq \lambda c_6 \l\|\hat{v}_n-\overline{u}_\lambda \r\|_r\quad \text{for some }c_6>0 \text{ and for all }n \in \N, 
	\end{align*}
	since $\overline{u} \geq 0$. Then, from the convergence properties in \eqref{15}, we conclude that
	\begin{align*}
		\limsup_{n\to\infty}\l[ \l\lan A_p\l(\hat{v}_n\r),\hat{v}_n-\overline{u}_\lambda\r\ran+\l\lan A_q\l(\hat{v}_n\r),\hat{v}_n-\overline{u}_\lambda\r\ran \r]\leq 0.
	\end{align*}
	By the monotonicity of $A_q$ we obtain
	\begin{align*}
		\limsup_{n\to\infty}\l[\l\lan A_p\l(\hat{v}_n\r),\hat{v}_n-\overline{u}_\lambda\r\ran+\l\lan A_q\l(\overline{u}_\lambda\r),\hat{v}_n-\overline{u}_\lambda\r\ran\r] \leq 0.
	\end{align*}
	Therefore,
	\begin{align*}
		\limsup_{n\to\infty} \l\lan A_p\l(\hat{v}_n\r),\hat{v}_n-\overline{u}_\lambda\r\ran\leq 0,
	\end{align*}
	which by Proposition \ref{proposition_1} implies that
	\begin{align}\label{16}
		\hat{v}_n\to \overline{u}_\lambda \quad\text{in }\Wpzero{p}.
	\end{align}
	
	From the Claim we know that $\hat{v}_1 \leq \hat{v}_n$ for all $n \in \N$ and so, $\hat{v}_1\leq \overline{u}_\lambda$. Thus, $\overline{u}_\lambda \neq 0$.
	
	For every $h \in\Wpzero{p}$, since $\hat{v}_1\in\interior$, by Hardy's inequality, we have
	\begin{align*}
		0 \leq \frac{|h(x)|}{\l[\hat{v}_n+\eps_n\r]^{\eta}} \leq |h(x)| \hat{v}_1^{-\eta} \in \Lp{1}\quad\text{ for all }n\in\N.
	\end{align*}
	
	Moreover, we have
	\begin{align*}
		\frac{h(x)}{\l[\hat{v}_n(x)+\eps_n\r]^{\eta}}\to \frac{h(x)}{\overline{u}_\lambda(x)^{\eta}} \quad\text{for a.\,a.\,}x\in\Omega
	\end{align*}
	due to \eqref{15}. Therefore, we can apply the Dominated Convergence Theorem and obtain
	\begin{align}\label{17}
		\into \frac{h}{\l[\hat{v}_n+\eps_n\r]^{\eta}}\,dx\to \into \frac{h}{\overline{u}_\lambda^{\eta}}\quad\text{as }n\to \infty.
	\end{align}
	
	We return to \eqref{12}, pass to the limit as $n\to\infty$ and use \eqref{16} as well as \eqref{17}. We obtain
	\begin{align*}
		\l\lan A_p\l(\overline{u}_\lambda\r),h\r\ran+\l\lan A_q\l(\overline{u}_\lambda\r),h\r\ran+\into f\l(x,\overline{u}_\lambda\r)h\,dx=\into \frac{\lambda h}{\overline{u}_\lambda^\eta}\,dx
	\end{align*}
	for all $h \in \Wpzero{p}$. Hence, $\overline{u}_\lambda$ is a positive solution of \eqref{problem2} for $\lambda>0$.
	
	From Marino-Winkert \cite{13-Marino-Winkert-2020} we have that
	\begin{align*}
		\hat{v}_n\in\Linf \quad\text{and}\quad \l\|\hat{v}_n\r\|_\infty \leq c_7
	\end{align*}
	for some $c_7>0$ and for all $n\in\N$. Then, by hypothesis H(i) we know that
	\begin{align*}
		\l\{N_f\l(\hat{v}_n\r)\r\}_{n\in\N} \subseteq \Linf \text{ is bounded}.
	\end{align*}

	We have
	\begin{align*}
		-\Delta_p \hat{v}_n-\Delta_q\hat{v}_n=\frac{\lambda}{\l[\hat{v}_n
		+\eps_n\r]^\eta}-f\l(x,\hat{v}_n\r) \quad \text{in }\Omega, \quad \hat{v}_n \big|_{\partial\Omega} =0
	\end{align*}
	for all $n \in\N$. 

	Using the nonlinear regularity theory of Lieberman \cite{11-Lieberman-1991}, we have that
	\begin{align*}
		\l\{\hat{v}_n\r\}_{n\in\N} \subseteq C^1_0(\close) \text{ is relatively compact}.
	\end{align*}
	Hence, due to \eqref{16}, we obtain $\hat{v}_n\to \overline{u}_\lambda$ in $C^1_0(\close)$. Since $\hat{v}_1 \leq \overline{u}_\lambda$, we then conclude that $\overline{u}_\lambda\in\interior$.
	
	So, we have proved that for every $\lambda>0$, problem \eqref{problem2} has a solution $\overline{u}_\lambda \in\interior$.
	
	We need to show that this is the unique positive solution of \eqref{problem2}. To this end, let $\overline{v}_\lambda \in \Wpzero{p}$ be another positive solution of \eqref{problem2}. Since $A_p$ and $A_q$ are strictly monotone and $f(x,\cdot)$ is nondecreasing, we have
	\begin{align*}
		0 &
		\leq \l\lan A_p\l(\overline{u}_\lambda\r)-A_p\l(\overline{v}_\lambda\r),\overline{u}_\lambda-\overline{v}_\lambda\r\ran+\l\lan A_p\l(\overline{u}_\lambda\r)-A_q\l(\overline{v}_\lambda\r),\overline{u}_\lambda-\overline{v}_\lambda\r\ran\\
		& \quad +\into \l[f\l(x,\overline{u}_\lambda\r)-f\l(x,\overline{v}_\lambda\r)\r] \l(\overline{u}_\lambda-\overline{v}_\lambda\r)\,dx\\
		& =\into \lambda \l[\frac{1}{\overline{u}_\lambda^\eta}-\frac{1}{\overline{v}_\lambda^\eta}\r] \l(\overline{u}_\lambda-\overline{v}_\lambda\r)\,dx \leq 0.
	\end{align*}
	Therefore, $\overline{u}_\lambda =\overline{v}_\lambda$.
	
	Finally, we are going to show the monotonicity of $\lambda \to \overline{u}_\lambda$. So, let $\lambda<\mu$. We consider the Carath\'eodory function $d_\mu\colon\Omega\times\R\to\R$ defined by
	\begin{align}\label{24}
		d_\mu(x,s)=
		\begin{cases}
			\mu \overline{u}_\lambda(x)^{-\eta}-f\l(x,\overline{u}_\lambda(x)\r)&\text{if } s \leq \overline{u}_\lambda(x),\\
			\mu s^{-\eta}-f(x,s)&\text{if }\overline{u}_\lambda(x)<s.
		\end{cases}
	\end{align}
	We set $D_\mu(x,s)=\int_0^s d_\mu(x,t)\,dt$ and consider the $C^1$-functional $\tau_\mu\colon\Wpzero{p}\to\R$ defined by
	\begin{align*}
		\tau_\mu(u) =\frac{1}{p}\|\nabla u\|_p^p+\frac{1}{q}\|\nabla u\|_q^q-\into D_\mu(x,u)\,dx \quad\text{for all }u \in \Wpzero{p}.
	\end{align*}
	Since $\tau_\mu\colon\Wpzero{p}\to\R$ is coercive, the direct method of the calculus of variations produces $\tilde{u}_\mu\in\Wpzero{p}$ such that
	\begin{align*}
		\tau_\mu\l(\tilde{u}_\mu\r)=\min \l[\tau_\mu(u)\,:\,u\in\Wpzero{p}\r].
	\end{align*}
	From \eqref{24} we see that
	\begin{align*}
		\tilde{u}_\mu \in K_{\tau_{\mu}}=\l\{u \in \Wpzero{p}\,:\,\tau_\mu'(u)=0\r\}\subseteq \l[\overline{u}_\lambda \r) \cap \interior
	\end{align*}
	and
	\begin{align*}
		\tilde{u}_\mu =\overline{u}_\mu \in \interior.
	\end{align*}
	Hence, $\overline{u}_\lambda \leq \overline{u}_\mu$.
\end{proof}

\section{Positive solutions} 

In this section we prove the existence and nonexistence of positive solutions for problem \eqref{problem} as $\lambda$ moves in $\overset{\circ}{\R}_+=(0,+\infty)$.

We introduce the following two sets
\begin{align*}
	\mathcal{L}&=\left\{\lambda>0: \text{problem \eqref{problem} has a positive solution}\right\},\\
	\mathcal{S}_\lambda&=\left\{u: u\text{ is a positive solution of problem \eqref{problem}}\right\}.
\end{align*}

\begin{proposition}\label{proposition_3}
	If hypotheses H hold, then $\overline{u}_\lambda \leq u$ for all $u \in \mathcal{S}_\lambda$.
\end{proposition}

\begin{proof}
	Let $u \in \mathcal{S}_\lambda$. We introduce the Carath\'eodory function $k_\lambda\colon\Omega\times \overset{\circ}{\R}_+\to \R$ defined by
	\begin{align}\label{25}
		k_\lambda(x,s)=
		\begin{cases}
			\lambda s^{-\eta}-f(x,s)&\text{if } 0<s\leq u(x),\\
			\lambda u(x)^{-\eta} -f(x,u(x))& \text{if }u(x)<s.
		\end{cases}
	\end{align} 
	We consider the following Dirichlet singular $(p,q)$-equation
	\begin{align}\label{26}
		-\Delta_p u-\Delta_qu=k_\lambda(x,u) \quad\text{in }\Omega, \quad u\big|_{\partial\Omega}=0, \quad u> 0.
	\end{align}
	Reasoning as in the proof of Proposition \ref{proposition_2}, see also Papageorgiou-R\u{a}dulescu-Repov\v{s} \cite[Proposition 10]{16-Papageorgiou-Radulescu-Repovs-2020}, we show that \eqref{26} has a positive solution $\tilde{u}_\lambda\in\interior$. The weak formulation of \eqref{26} is given by
	\begin{align}\label{27}
		\l\lan A_p\l(\tilde{u}_\lambda\r),h\r\ran+\l\lan A_q\l(\tilde{u}_\lambda\r),h\r\ran=\into k_\lambda \l(x,\tilde{u}_\lambda\r)h\,dx\quad\text{for all }u \in \Wpzero{p}.
	\end{align}
	Now, we choose $h=\l(\tilde{u}_\lambda-u\r)^+\in\Wpzero{p}$ as test function in \eqref{27}. Then, by applying \eqref{25}, $u\geq 0$ and $u\in \mathcal{S}_\lambda$, we obtain
	\begin{align*}
		& \l\lan A_p\l(\tilde{u}_\lambda\r),\l(\tilde{u}_\lambda-u\r)^+\r\ran+\l\lan A_q\l(\tilde{u}_\lambda\r),\l(\tilde{u}_\lambda-u\r)^+\r\ran\\
		& =\into \l[\lambda u^{-\eta}-f(x,u)\r]\l(\tilde{u}_\lambda-u\r)^+\,dx\\
		& \leq \l[\lambda \l(u^{-u}+u^{\theta-1}\r)-f(x,u)\r]\l(\tilde{u}_\lambda-u\r)^+\,dx\\
		&=\l\lan A_p\l(u\r),\l(\tilde{u}_\lambda-u\r)^+\r\ran+\l\lan A_q\l(u\r),\l(\tilde{u}_\lambda-u\r)^+\r\ran.
	\end{align*}
	Therefore, $\tilde{u}_\lambda \leq u$ because of the monotonicity of $A_p$ and $A_q$.
	
	Then, from \eqref{25} and Proposition \ref{proposition_2}, it follows that $\tilde{u}_\lambda =\overline{u}_\lambda \in \interior$ and so, $\overline{u}_\lambda \leq u$ for all $u \in \mathcal{S}_\lambda$.
\end{proof}

Next we determine the regularity of the elements of the solution set of $\mathcal{S}_\lambda$.

\begin{proposition}\label{proposition_4}
	If hypotheses H hold, then $\mathcal{S}_\lambda \subseteq \interior$ for all $\lambda>0$.
\end{proposition}

\begin{proof}
	The result is trivially true if $\mathcal{S}_\lambda =\emptyset$. So, suppose that $\mathcal{S}_\lambda \neq\emptyset$ and let $u \in \mathcal{S}_\lambda$. From Proposition \ref{proposition_3} we know that $\overline{u}_\lambda \leq u$ and so $u^{-\eta} \leq \overline{u}_\lambda^{-\eta}\in\Lp{1}$. Recall that $\hat{v}_1\leq \overline{u}_\lambda$ and $\hat{v}_1^{-\eta} \in \Lp{1}$, see the proof of Proposition \ref{proposition_2}.
	Therefore, using Theorem B.1 of Giacomoni-Saoudi \cite{Giacomoni-Saoudi-2010}, we see that $u \in C^1_0(\close)_+\setminus\{0\}$.

	On account of hypotheses H(i), (ii), if $\rho=\|u\|_\infty$, then we can find $\hat{\xi}_\rho>0$ such that 
	\begin{align*}
		\hat{\xi}_\rho s^{p-1}-f(x,s) \geq 0\quad\text{for a.a.\,}x\in\Omega \text{ and for all }0\leq s\leq \rho.
	\end{align*}
	Using this, we have 
	\begin{align*}
		\Delta_p u +\Delta_q u \leq \hat{\xi}_\rho u^{p-1} \quad \text{in }\Omega.
	\end{align*}
	Then, by Pucci-Serrin \cite[pp.\,111 and 120]{21-Pucci-Serrin-2007}, we derive $u \in \interior$. Hence, $\mathcal{S}_\lambda\subseteq \interior$.
\end{proof}

Next, we are going to prove the nonemptiness of $\mathcal{L}$.

\begin{proposition}
	If hypotheses H hold, then $\mathcal{L}\neq \emptyset$.
\end{proposition}

\begin{proof}
	Let $\overline{u}_\lambda\in\interior$ be the unique positive solution of \eqref{problem2}, see Proposition \ref{proposition_2}. We introduce the Carath\'eodory function $e_\lambda \colon\Omega \times \R\to\R$ defined by
	\begin{align}\label{29}
		e_\lambda (x,s)=
		\begin{cases}
			\lambda \overline{u}_\lambda (x)^{-\eta}-f\l(x,\overline{u}_\lambda(x)\r)+\lambda \l(s^+\r)^{\theta-1}&\text{if } s \leq \overline{u}_\lambda(x),\\
			\lambda s^{-\eta} -f(x,s)+\lambda s^{\theta-1}&\text{if }\overline{u}_\lambda(x)<s.
		\end{cases}
	\end{align}
	We set $E_\lambda(x,s)=\int^s_0e_\lambda(x,t)\,dt$ and consider the functional $\gamma_\lambda\colon\Wpzero{p}\to\R$ defined by
	\begin{align*}
		\gamma_\lambda(u)=\frac{1}{p}\|\nabla u\|_p^p+\frac{1}{q}\|\nabla u\|_q^q-\into E_\lambda(x,u)\,dx\quad\text{for all }u\in\Wpzero{p}.
	\end{align*}
	Since $\overline{u}_\lambda^{-\eta}\in \Lp{s}$ for $s>N$, see the proof of Proposition \ref{proposition_2}, we have that $\gamma_\lambda \in C^1(\Wpzero{p})$, see also Proposition 3 of Papageorgiou-Smyrlis \cite{17-Papageorgiou-Smyrlis-2015}.
	
	From \eqref{29} and hypothesis H(ii), we infer that $\gamma_\lambda$ is coercive. Moreover, it is also sequentially weakly lower semicontinuous. Hence, there exists a global minimizer $u^\circ_\lambda \in \Wpzero{p}$ of $\gamma_\lambda$, that is,
	\begin{align}\label{30}
		\gamma_\lambda \l(u^\circ_\lambda\r)=\min\l[\gamma_\lambda(u)\,:\,u\in\Wpzero{p}\r].
	\end{align}
	
	Let $u \in \interior$ and choose $t \in (0,1)$ small so that $tu \leq \overline{u}_\lambda$. Recall that $\overline{u}_\lambda\in\interior$ and use Proposition 4.1.22 of Papageorgiou-R\u{a}dulescu-Repov\v{s} \cite[p.\,274]{15-Papageorgiou-Radulescu-Repovs-2019}.
	
	We have 
	\begin{align}\label{31}
		\gamma_\lambda(tu) \leq \frac{t^p}{p} \|\nabla u\|_p^p+\frac{t^q}{q}\|\nabla u\|_q^q-t \into \l[\lambda \overline{u}_\lambda^{-\eta}-f\l(x,\overline{u}_\lambda\r)\r]u\,dx.
	\end{align}
	
	Let $\lambda_0 =\l\|\overline{u}_\lambda^{\eta} f\l(x,\overline{u}_\lambda\r)\r\|_\infty$, see hypothesis H(i), and let $\lambda>\lambda_0$. Then
	\begin{align*}
		\into \l[\lambda \overline{u}_\lambda^{-\eta}-f\l(x,\overline{u}_\lambda\r)\r]\,dx=\overline{\mu}>0.
	\end{align*}
	So, from \eqref{31} we have 
	\begin{align*}
		\gamma_\lambda(tu) \leq c_{10}t^q-\overline{\mu} t\quad\text{for some } c_{10}>0,
	\end{align*}
	since $t\in (0,1)$ and $q<p$.
	
	Since $q>1$, by taking $t\in(0,1)$ even smaller if necessary, we see that $\gamma_\lambda (tu)<0$. Taking \eqref{30} into account we know that
	\begin{align*}
		\gamma_\lambda \l(u^\circ_\lambda\r)<0=\gamma_\lambda (0)\quad\text{for all }\lambda >\lambda_0.
	\end{align*}
	Thus, $u^\circ_\lambda\neq 0$. 
	
	From \eqref{30} we have $\gamma_\lambda'\l(u^\circ_\lambda\r)=0$, that is,
	\begin{align}\label{32}
		\l\lan A_p\l(u^\circ_\lambda\r),h\r\ran+\l\lan A_q\l(u^\circ_\lambda\r),h\r\ran=\into e_\lambda \l(x,u^\circ_\lambda\r)h\,dx\quad\text{ for all }h\in\Wpzero{p}.
	\end{align}
	We choose $h=\l(\overline{u}_\lambda-u^\circ_\lambda\r)^+\in\Wpzero{p}$ as test function in \eqref{32}. Applying \eqref{29} and Proposition \ref{proposition_2} gives
	\begin{align*}
		& \l\lan A_p\l(u^\circ_\lambda\r),\l(\overline{u}_\lambda-u^\circ_\lambda\r)^+\r\ran+\l\lan A_q\l(u^\circ_\lambda\r),\l(\overline{u}_\lambda-u^\circ_\lambda\r)^+\r\ran\\
		& =\into \l[\lambda \overline{u}_\lambda^{-\eta}-f\l(x,\overline{u}_\lambda\r)+\lambda \l(\l(u^\circ_\lambda\r)^+\r)^{\theta-1}\r] \l(\overline{u}_\lambda-u^\circ_\lambda\r)^+\,dx\\
		& \geq \into \l[\lambda \overline{u}_\lambda^{-\eta}-f\l(x,\overline{u}_\lambda\r)\r] \l(\overline{u}_\lambda-u^\circ_\lambda\r)^+\,dx\\
		& = \l\lan A_p\l(\overline{u}_\lambda\r),\l(\overline{u}_\lambda-u^\circ_\lambda\r)^+\r\ran+\l\lan A_q\l(\overline{u}_\lambda\r),\l(\overline{u}_\lambda-u^\circ_\lambda\r)^+\r\ran.
	\end{align*}
	As before, by the monotonicity of $A_p$ and $A_q$ we conclude that $\overline{u}_\lambda \leq u^\circ_\lambda$. Using this fact along with \eqref{29} and \eqref{32} we infer that
	\begin{align*}
		u^\circ_\lambda \in \mathcal{S}_\lambda\subseteq \interior,
	\end{align*}
	see Proposition \ref{proposition_4}. Therefore, $\lambda \in \mathcal{L}$ and so $(\lambda_0,+\infty) \subseteq \mathcal{L}\neq\emptyset$.
\end{proof}

The next proposition establishes a structural property for $\mathcal{L}$, namely that $\mathcal{L}$ is an upper half-line.

\begin{proposition}\label{proposition_6}
	If hypotheses H hold, $\lambda \in\mathcal{L}$ and $\mu>\lambda$, then $\mu \in \mathcal{L}$.
\end{proposition}

\begin{proof}
	Since $\lambda  \in \mathcal{L}$ there exists $u_\lambda \in\mathcal{S}_\lambda \subseteq \interior$, see Proposition \ref{proposition_4}. From Proposition \ref{proposition_3} we have $\overline{u}_\lambda \leq u_\lambda$. Therefore,
	\begin{align}\label{33}
		u_\lambda^{-\eta} \in \Lp{s} \quad\text{for }s>N.
	\end{align}
	We now introduce the Carath\'eodory function $g_\mu \colon\Omega \times \R\to\R$ defined by
	\begin{align}\label{34}
		g_\mu(x,s)=
		\begin{cases}
			\mu \l[u_\lambda(x)^{-\eta}+u_\lambda(x)^{\theta-1}\r]-f\l(x,u_\lambda(x)\r)  &\text{if } s \leq u_\lambda(x),\\
			\mu \l[s^{-\eta}+s^{\theta-1}\r]-f\l(x,s\r) &\text{if }u_\lambda(x)<s.
		\end{cases}
	\end{align}
	We set $G_\mu(x,s)=\int^s_0g_\mu(x,t)\,dt$ and consider the $C^1$-functional $\ph_\mu\colon\Wpzero{p}\to\R$ defined by
	\begin{align*}
		\ph_\mu(u)=\frac{1}{p}\|\nabla u\|_p^p+\frac{1}{q}\|\nabla u\|_q^q-\into G_\mu(x,u)\,dx\quad\text{for all }u\in\Wpzero{p},
	\end{align*}
	see \eqref{33}. 
	
	From \eqref{33} and hypothesis H(ii) we see that $\ph_\mu$ is coercive and we know it is also sequentially weakly lower semicontinuous. Hence, we can find $u_\mu \in\Wpzero{p}$ such that 
	\begin{align*}
		\ph_\mu \l(u_\mu\r)=\min\l[\ph_\mu(u)\,:\,u\in\Wpzero{p}\r].
	\end{align*}
	This implies that $\ph_\mu'\l(u_\mu\r)=0$, that is,
	\begin{align}\label{35}
		\l\lan A_p\l(u_\mu\r),h\r\ran+\l\lan A_q\l(u_\mu\r),h\r\ran=\into g_\mu \l(x,u_\mu\r)h\,dx\quad\text{ for all }h\in\Wpzero{p}.
	\end{align}
	We choose $h=\l(u_\lambda-u_\mu\r)^+\in\Wpzero{p}$ as test function in \eqref{35}. Applying \eqref{34}, $\lambda<\mu$ and $u_\lambda \in \mathcal{S}_\lambda$, we obtain	
	\begin{align*}
	& \l\lan A_p\l(u_\mu\r),\l(u_\lambda-u_\mu\r)^+\r\ran+\l\lan A_q\l(u_\mu\r),\l(u_\lambda-u_\mu\r)^+\r\ran\\
	& =\into \l[\mu \l(u_\lambda^{-\eta}+u_\lambda^{\theta-1}\r)-f\l(x,u_\lambda\r)\r] \l(u_\lambda-u_\mu\r)^+\,dx\\
	& \geq \into \l[\lambda \l(u_\lambda^{-\eta}+u_\lambda^{\theta-1}\r)-f\l(x,u_\lambda\r)\r] \l(u_\lambda-u_\mu\r)^+\,dx\\
	& = \l\lan A_p\l(u_\lambda\r),\l(u_\lambda-u_\mu\r)^+\r\ran+\l\lan A_q\l(u_\lambda\r),\l(u_\lambda-u_\mu\r)^+\r\ran.
	\end{align*}
	Again, from the monotonicity of $A_p$ and $A_q$, we deduce that $u_\lambda \leq u_\mu$. This along with \eqref{34} as well as \eqref{35} implies that $u_\mu \in \mathcal{S}_\mu \subseteq \interior$. Hence, $\mu \in \mathcal{L}$.
\end{proof}

So, according to Proposition \ref{proposition_6}, $\mathcal{L}$ is an upper half-line. Moreover, a byproduct of the proof of Proposition \ref{proposition_6} is the following corollary.

\begin{corollary}\label{corollary_7}
	If hypotheses H hold, $\lambda \in\mathcal{L}$, $u_\lambda\in\mathcal{S}_\lambda$ and $\mu>\lambda$, then $\mu\in\mathcal{L}$ and there exists $u_\mu \in \mathcal{S}_\mu$ such that $u_\lambda \leq u_\mu$.
\end{corollary}

If we strengthen a little the conditions on $f(x,\cdot)$, we can improve the assertion of this corollary.

\begin{enumerate}[leftmargin=1.2cm]
	\item[H':]
	$f\colon\Omega \times \R\to \R$ is a Carath\'{e}odory function such that $f(x,0)=0$ for a.\,a.\,$x\in\Omega$, $f(x,\cdot)$ is nondecreasing, hypotheses H'(i), (ii), (iii) are the same as the corresponding hypotheses H(i), (ii), (iii) and
	\begin{enumerate}[itemsep=0.2cm, topsep=0.2cm]
		\item[(iv)]
			for every $\varrho>0$ there exists $\hat{\xi}_\varrho>0$ such that the function
			\begin{align*}
				s \to \hat{\xi}_\varrho s^{p-1}- f(x,s)
		\end{align*}
		is nondecreasing on $[0,\varrho]$ for a.\,a.\,$x\in\Omega$.
	\end{enumerate}
\end{enumerate}

\begin{remark}
	The examples in Section \ref{section_2} satisfy this extra condition.
\end{remark}

\begin{proposition}
	If hypotheses H' hold, $\lambda \in\mathcal{L}$, $u_\lambda\in\mathcal{S}_\lambda$ and $\mu>\lambda$, then $\mu\in\mathcal{L}$ and there exists $u_\mu \in \mathcal{S}_\mu$ such that $u_\mu-u_\lambda  \in \interior$.
\end{proposition}

\begin{proof}
	From Corollary \ref{corollary_7} we already know that $\mu \in\mathcal{L}$ and we can find $u_\mu\in\mathcal{S}_\mu\subseteq \interior$ such that
	\begin{align}\label{36}
		u_\lambda \leq u_\mu.
	\end{align}
	Let $\varrho = \|u_\mu\|_\infty$ and let $\hat{\xi}_\varrho>0$ be as postulated by hypothesis H'(iv). Since $\lambda <\mu$, $u_\lambda \in \mathcal{S}_\lambda$ and due to \eqref{36} as well as hypothesis H'(iv) we obtain
	\begin{align}\label{37}
		\begin{split}
			&-\Delta_p u_\lambda -\Delta_q u_\lambda +\hat{\xi}_\varrho u_\lambda^{p-1}-\mu u_\lambda^{-\eta}\\
			&\leq -\Delta_p u_\lambda -\Delta_q u_\lambda +\hat{\xi}_\varrho u_\lambda^{p-1}-\lambda u_\lambda^{-\eta}\\
			&=\lambda u_\lambda^{\theta-1}+\hat{\xi}_\varrho u_\lambda^{p-1}-f(x,u_\lambda)\\
			&\leq \mu u_\mu^{\theta-1}+\hat{\xi}_\varrho u_\mu^{p-1}-f(x,u_\mu)\\
			& = -\Delta_p u_\mu -\Delta_q u_\mu +\hat{\xi}_\varrho u_\mu^{p-1}-\mu u_\mu^{-\eta}.
		\end{split}
	\end{align}
	Note that since $u_\lambda \interior$ we have
	\begin{align*}
		0 \prec \l[\mu-\lambda\r] u_\lambda^{\theta-1}.
	\end{align*}
	So, from \eqref{37} and Proposition 7 of Papageorgiou-R\u{a}dulescu-Repov\v{s} \cite{16-Papageorgiou-Radulescu-Repovs-2020}, we conclude that $u_\mu-u_\lambda\in\interior$.
\end{proof}

Let $\lambda_*=\inf \mathcal{L}$. 

\begin{proposition}
	If hypotheses H' hold, then $\lambda_*>0$.
\end{proposition}

\begin{proof}
	On account of hypotheses H'(ii), (iii) we can find $\hat{\lambda}>0$ such that
	\begin{align}\label{38}
		\hat{\lambda}s^{\theta-1}-f(x,s) \leq 0 \quad\text{for a.a.\,}x\in\Omega \text{ and for all }s \geq 0.
	\end{align}
	Consider $\lambda \in (0,\hat{\lambda})$ and suppose that $\lambda\in\mathcal{L}$. Then we can find $u_\lambda \in \mathcal{S}_\lambda\subseteq \interior$. We set $\varrho_\lambda=\max_{\close} u_\lambda$. Then, for $\delta\in (0,\varrho_\lambda)$ small enough, we set $ \varrho_\lambda^\delta=\varrho_\lambda-\delta>0$. For $\hat{\xi}_\lambda=\hat{\xi}_{\varrho_\lambda}>0$ as postulated by hypothesis H'(iv) along with \eqref{38}, $\hat{\lambda}>\lambda$, $u_\lambda\in\mathcal{S}_\lambda$ and $\delta>0$ small enough, we obtain
	\begin{align*}
		& -\Delta_p \varrho_\lambda^\delta-\Delta_q \varrho_\lambda^\delta+\hat{\xi}_\lambda \l(\varrho_\lambda^\delta\r)^{p-1}-\lambda \l(\varrho_\lambda^\delta\r)^{-\eta}\\
		&\geq \hat{\xi}_\lambda \varrho_\lambda^{p-1}-\chi(\delta)\quad\text{with }\chi(\delta)\to 0^+ \text{ as }\delta\to 0^+\\
		& \geq \hat{\lambda} \varrho_\lambda^{\theta-1} -f\l(x,\varrho_\lambda\r)+\hat{\xi}_\lambda \varrho_\lambda^{p-1}-\chi(\delta)\\
		& = \lambda \varrho_\lambda^{\theta-1}-f\l(x,\varrho_\lambda\r)+\hat{\xi}_\lambda \varrho_\lambda^{p-1}+\l[\hat{\lambda}-\lambda \r] \varrho_\lambda^{\theta-1} -\chi(\delta)\\
		& \geq  \lambda \varrho_\lambda^{\theta-1}-f\l(x,\varrho_\lambda\r)+\hat{\xi}_\lambda \varrho_\lambda^{p-1}\\
		& \geq \lambda u_\lambda^{\theta-1}-f\l(x,u_\lambda\r)+\hat{\xi}_\lambda u_\lambda^{p-1}\\
		& = -\Delta_p u_\lambda-\Delta_q u_\lambda+\hat{\xi}_\lambda u_\lambda^{p-1}-\lambda u_\lambda^{-\eta}.
	\end{align*}
	Invoking Proposition 6 of Papageorgiou-R\u{a}dulescu-Repov\v{s} \cite{16-Papageorgiou-Radulescu-Repovs-2020}, we have that
	\begin{align*}
		\varrho_\lambda^\delta > u_\lambda(x) \quad\text{ for all }x\in\close \text{ and for all small }\delta \in (0,\varrho_\lambda),
	\end{align*}
	a contradiction to the definition of $\varrho_\lambda$. Therefore
	\begin{align*}
		0<\hat{\lambda} \leq \lambda_* =\inf \mathcal{L}.
	\end{align*}
\end{proof}

Next, we show that $\lambda_*$ is admissible, that is, $\lambda_*>0$.

\begin{proposition}
	If hypotheses H' hold, then $\lambda_* \in \mathcal{L}$.
\end{proposition}

\begin{proof}
	Let $\{\lambda_n\}_{n\in\N} \subseteq \mathcal{L}$ be such that $\lambda_n \searrow \lambda_*$. For every $n \in \N$, let $u_n \in \mathcal{S}_{\lambda_n}\subseteq \interior$. From Proposition \ref{proposition_2} we know that
	\begin{align}\label{39}
		\overline{u}_{\lambda_*} \leq u_n\quad\text{for all }n \in \N.
	\end{align}
	Moreover we have
	\begin{align}\label{40}
		\l\lan A_p\l(u_n\r),h\r\ran+\l\lan A_q\l(u_n\r),h\r\ran =\into \l[\lambda_n \l(u_n^{-\eta}+u_n^{\theta-1}\r)-f\l(x,u_n\r)\r]h\,dx
	\end{align}
	for all $h \in\Wpzero{p}$ and for all $n \in \N$. 
	
	On account of hypotheses H'(i), (ii), (iii) there exists $c_{11}>0$ such that
	\begin{align}\label{41}
		\lambda_n s^{\theta-1} -f(x,s) \leq c_{11}
	\end{align}
	for a.\,a.\,$x\in\Omega$, for all $s\geq 0$ and for all $n\in\N$.

	Choosing $h =u_n \in \Wpzero{p}$ in \eqref{40} and using \eqref{39} and \eqref{41}, results in
	\begin{align*}
		\l\|u_n\r\|^p \leq c_{12} \l\|u_n\r\|\quad\text{for some }c_{12}>0 \text{ and for all }n\in\N.
	\end{align*}
	Therefore, $\{u_n\}_{n\in\N}\subseteq \Wpzero{p}$ is bounded. 
	
	So, we may assume that 
	\begin{align}\label{42}
		u_n\weak u_*\quad\text{in }\Wpzero{p}\quad\text{and}\quad u_n\to u_* \quad\text{in }\Lp{r}.
	\end{align}
	Taking $h=u_n-u_*\in\Wpzero{p}$ as test function in \eqref{40}, passing to the limit as $n\to\infty$ and using \eqref{42} yields
	\begin{align*}
		\limsup_{n\to\infty} \l\lan A_p\l(u_n\r),u_n-u_*\r\ran \leq 0,
	\end{align*}
	see the proof of Proposition \ref{proposition_2}. Then, from Proposition \ref{proposition_1} we conclude that
	\begin{align}\label{43}
		u_n\to u_*\quad\text{in }\Wpzero{p}.
	\end{align}
	Now we can apply \eqref{43} along with \eqref{39}  as well as \eqref{40}, as in the proof of Proposition \ref{proposition_2}, in the limit as $n\to\infty$, we obtain
	\begin{align*}
		\overline{u}_{\lambda_*}\leq u_*
	\end{align*}
	and
	\begin{align*}
		\l\lan A_p\l(u_*\r),h\r\ran+\l\lan A_q\l(u_*\r),h\r\ran =\into \l[\lambda_* \l(u_*^{-\eta}+u_*^{\theta-1}\r)f\l(x,u_*\r)\r]h\,dx
	\end{align*}
	for all $h \in\Wpzero{p}$. Finally, we reach $u_* \in \mathcal{S}_{\lambda_*}\subseteq \interior$ and so $\lambda_* \in \mathcal{L}$.
\end{proof}

So, we have $\mathcal{L}=[\lambda_*,+\infty)$ and we can state the following theorem for the positive solutions of problem \eqref{problem}.

\begin{theorem}
	If hypotheses H' hold, then there exists $\lambda_*>0$ such that
	\begin{enumerate}
		\item
			for every $\lambda \geq \lambda_*$, problem \eqref{problem} has a positive solution $u_\lambda \in \interior$;
		\item
			for every $\lambda \in (0,\lambda_*)$, problem \eqref{problem} has no positive solutions.
	\end{enumerate}
\end{theorem}

\section{title}

If any of the sections are not relevant to your manuscript, please include the heading and write 'Not applicable' for that section.

\end{document}